\documentclass{article}
\usepackage{amsmath,amssymb,amsthm,mysty_base}
\usepackage[dvipdfmx]{graphicx}

\theoremstyle{plain}
\newtheorem{theorem}{Theorem}[section]

\theoremstyle{definition}
\newtheorem{definition}{Definition}[section]
\newtheorem{lemma}[definition]{Lemma}
\newtheorem{remark}[definition]{Remark}
\newtheorem{corollary}[definition]{Corollary}

\begin{document}

\title{A Proof of Basic Limit Theorem of Renewal Theory
}
\markright{A Proof of Basic Limit Theorem of Renewal Theory}
\author{Toshihiro Koga\thanks{toshihiro\_f\_ma\_mgkvv@t01.itscom.net}}

\maketitle

\begin{abstract}
Let $\{q_n\}_{n=0}^\infty\subset [0,1]$ satisfy $q_0=0$, $\sum_{n=0}^\infty q_n=1$, 
and $\gcd\{n\geq 1\mid q_n\neq 0\}=1$. 
We consider the following process: Let $x$ be a real number. We first set $x=0$. 
Then $x$ is increased by $i$ with probability $q_i~(i=0,1,2,\cdots)$ every time. 
For $n\geq 0$, let $p_n$ be the probability such that $x=n$ occurs, 
so we have $p_0=1$ and $p_n=q_1p_{n-1}+q_2p_{n-2}+\cdots+q_np_0~(n\geq 1)$. 
In this setting, we have $\lim_n p_n=1/\sum_{i=0}^\infty iq_i$, 
where we define $1/\sum_{i=0}^\infty iq_i=0$ if $\sum_{i=0}^\infty iq_i=+\infty$. 
This result is known as (discrete case of) Blackwell's renewal theorem. 
The proof of $\lim_n p_n=1/\sum_{i=0}^\infty iq_i$ is not trivial, 
while the meaning of $\lim_n p_n=1/\sum_{i=0}^\infty iq_i$ is clear 
since the expected value of increasing number $i$ is $\sum_{i=0}^\infty iq_i$. 
Many proofs of this result have been given. 
In this paper, we will also provide a proof of this result. 
The idea of our proof is based on Fourier-analytic methods 
and Tauberian theorems for almost convergent sequences, 
while we actually need only elementary analysis. 
\end{abstract}

\section{Introduction.}
In renewal theory, the following theorem was first proved by \cite{erdos1949theorem}, 
and was called Basic Limit Theorem (e.g., \cite{feller1}): 
\begin{theorem}\label{main_theo}
Let $q:\Zp\to \R$ satisfy 
\begin{align}\label{q_defi}
q_n\geq 0~(n\geq 0),~q_0=0,~\sum_{n=0}^\infty q_n=1,~\gcd\mk{n\geq 1\mid q_n\neq 0}=1. 
\end{align}
Let $p:\Zp\to \R$ be defined as 
\begin{align}\label{p_defi}
p_0:=1,~p_n:=\sum_{i=0}^n q_ip_{n-i}=\sum_{i=1}^n q_ip_{n-i}~(n\geq 1). 
\end{align}
Then $\lim_n p_n=1/\sum_{i=0}^\infty iq_i$, 
where we define $1/\sum_{i=0}^\infty iq_i=0$ if $\sum_{i=0}^\infty iq_i=+\infty$. 
\end{theorem}
Later, Theorem \ref{main_theo} was generalized to Blackwell's renewal theorem (BRT)
\cite{blackwell1948}, \cite{blackwell1953}. 
Today, many proofs of BRT are known 
(elementary methods, Fourier-analytic methods, probabilistic methods, and so on). 
Some references can be found in \cite{dippon2005simplified}, 
where \cite{dippon2005simplified} itself provides a proof of BRT. 
In this paper, we also provide another proof of (discrete case of) BRT, i.e., 
we give a proof of Theorem \ref{main_theo}. 
In fact, the author already provided a proof of Theorem \ref{main_theo} in \cite{koga2021tauberian}. 
The proof of \cite{koga2021tauberian} requires delicate analysis of Wiener-Ikehara Tauberian theorem, 
while the proof we provide here is based on Fourier-analytic methods 
and Tauberian theorems for almost convergent sequences. 
As for this topic, we also refer to \cite{pinsky1976note}: 
Theorem \ref{main_theo}, with additional conditions  
$\sum_{i=0}^\infty iq_i<+\infty$ and $\sum_{i=1}^\infty (i\log i)q_i<+\infty$, 
is proved in \cite{pinsky1976note} via elementary Fourier-analytic method. 
The general case can not be obtained by the method in \cite{pinsky1976note}, 
but if we combine it with Tauberian methods for almost convergent sequences, 
then we obtain the general case, i.e., the entire Theorem \ref{main_theo}, as described below. 

Throughout this paper, let $q:\Zp\to \R$ always satisfy (\ref{q_defi}), 
and $p:\Zp\to \R$ always satisfy (\ref{p_defi}). We also define $Q:\Zp\to [0,1]$ as 
$Q_n:=\sum_{i=n+1}^\infty q_i\in [0,1]~~(n\geq 0)$. 

\section{Basic properties. }
In this section, we provide basic properties of $q$, $p$, $Q$. 

\begin{definition}
For $a,b:\Zp\to \C$, we define $a*b:\Zp\to \C$ as \\
$(a*b)_n:=\sum_{i=0}^n a_ib_{n-i}~(n\geq 0)$. 
For $a:\Zp\to \C$, we define $\Delta[a]:\Zp\to \C$ as $\Delta[a]_0:=a_0,~\Delta[a]_n:=a_n-a_{n-1}~(n\geq 1)$. 
We also define $\Delta^k[a]:\Zp\to \C~(k\geq 0)$ as $\Delta^0[a]:=a,~\Delta^k[a]:=\Delta[\Delta^{k-1}[a]]~(k\geq 1)$. 
Next, we define $I:\Zp\to \C$ as $I_0:=1,~I_n:=0~(n\geq 1)$. 
We also define $1:\Zp\to \C$ as $1_n:=1~(n\geq 0)$. 
Next, $a:\Zp\to \C$ is said to be bounded iff $|a_n|\leq C~(\forall n\geq 0)$ for some constant $C>0$. 
\end{definition}

\begin{lemma}
We have $0\leq p_n\leq 1~(n\geq 0)$ and $p=I+p*q$. 
Moreover, $\sum_{i=0}^\infty Q_i=\sum_{i=0}^\infty iq_i$, $(I-q)*1=Q$, and $p*Q=1$. 
\end{lemma}

\begin{proof}
Since $p_0=1$ and $p_n=\sum_{i=1}^n q_ip_{n-i}~(n\geq 1)$, 
we can easily show $0\leq p_n\leq 1~(n\geq 0)$ by induction on $n\geq 0$. 
Next, we can easily verify $p=I+p*q$. 
Next, $\sum_{i=0}^\infty Q_i=Q_0+Q_1+Q_2+\cdots=(q_1+q_2+q_3+\cdots)+(q_2+q_3+\cdots)+(q_3+\cdots)
=1q_1+2q_2+3q_3+\cdots=\sum_{i=1}^\infty iq_i=\sum_{i=0}^\infty iq_i$. 
Next, we have 
$((I-q)*1)_n=\sum_{i=0}^n (I-q)_i1_{n-i}=\sum_{i=0}^n (I-q)_i=(I_0-q_0)+\sum_{i=1}^n (-q_i)=1-\sum_{i=1}^nq_i
=\sum_{i=1}^\infty q_i-\sum_{i=1}^nq_i=\sum_{i=n+1}^\infty q_i=Q_n~(n\geq 0)$, 
so $(I-q)*1=Q$. Since $p=I+p*q$, we have $p*(I-q)=p*I-p*q=p-p*q=I$, i.e., 
$p*(I-q)=I$, so $p*(I-q)*1=I*1=1$. Thus we obtain $p*Q=1$. 
\end{proof}

\section{Tauberian methods.}
In this section, we provide some crucial lemmas. 
The idea of following lemmas is based on Tauberian theorems for almost convergent sequences. 

\begin{lemma}\label{dep_lemm}
If $\lim_n \Delta[p]_n=0$, then we have $\lim_n p_n=1/\sum_{i=0}^\infty iq_i$. 
\end{lemma}

\begin{proof}
STEP1: Let $\ep>0$ and $M\geq 1$ be arbitrary. 
Since $\lim_n (p_n-p_{n-1})=0$, we have $\lim_n (p_n-p_{n-2})=\lim_n ((p_n-p_{n-1})+(p_{n-1}-p_{n-2}))=0+0$. 
Inductively, we have $\lim_n (p_n-p_{n-i})=0$ for any $i\geq 0$. 
Then, for any $0\leq i\leq M$, there exists $n_i\geq i$ such that $|p_n-p_{n-i}|<\ep~(n\geq n_i)$. 
Let $N:=\max\mk{M,n_0,\cdots,n_M}$. Then $|p_n-p_{n-i}|<\ep~(n\geq N,~0\leq i\leq M)$. 
Since $p*Q=1$, we have $\sum_{i=0}^n Q_ip_{n-i}=1~(n\geq 0)$. 
In particular, $\sum_{i=0}^M Q_ip_{n-i}\leq \sum_{i=0}^n Q_ip_{n-i}=1~(n\geq N)$. 
Since $p_n-p_{n-i}\leq \ep~(n\geq N,~0\leq i\leq M)$, we have 
$\sum_{i=0}^M Q_i(p_n-\ep)\leq \sum_{i=0}^M Q_ip_{n-i}\leq 1~(n\geq N)$, i.e., 
$p_n-\ep\leq 1/\sum_{i=0}^M Q_i~(n\geq N)$. Then $\limsup_n p_n-\ep\leq 1/\sum_{i=0}^M Q_i$. 
Since $\ep>0$ is arbitrary, we have $\limsup_n p_n\leq 1/\sum_{i=0}^M Q_i$. 
Taking $M\to +\infty$, we obtain $\limsup_n p_n\leq 1/\sum_{i=0}^\infty Q_i$. \\
STEP2: In this step, assume that $\sum_{i=0}^\infty Q_i<+\infty$. 
Let $\ep>0$ and $M\geq 1$ be arbitrary. 
By the same argument as above, there exists $N\geq M$ such that 
$|p_n-p_{n-i}|<\ep~(n\geq N,~0\leq i\leq M)$. 
Since $\sum_{i=0}^n Q_ip_{n-i}=1~(n\geq 0)$, we have
$1=\sum_{i=0}^n Q_ip_{n-i}=\sum_{i=0}^M Q_ip_{n-i}+\sum_{i=M+1}^n Q_ip_{n-i}
\leq \sum_{i=0}^M Q_ip_{n-i}+\sum_{i=M+1}^n Q_i
\leq \sum_{i=0}^M Q_ip_{n-i}+\sum_{i=M+1}^\infty Q_i~(n\geq N)$, i.e., 
$1-\sum_{i=M+1}^\infty Q_i\leq \sum_{i=0}^M Q_ip_{n-i}~(n\geq N)$. 
Since $p_n-p_{n-i}\geq -\ep~(n\geq N,~0\leq i\leq M)$, we have 
$\sum_{i=0}^M Q_ip_{n-i}\leq \sum_{i=0}^M Q_i(p_n+\ep)~(n\geq N)$, so 
$1-\sum_{i=M+1}^\infty Q_i\leq \sum_{i=0}^M Q_i(p_n+\ep)~(n\geq N)$, i.e., 
$(1-\sum_{i=M+1}^\infty Q_i)/\sum_{i=0}^M Q_i\leq p_n+\ep~(n\geq N)$. 
Then $(1-\sum_{i=M+1}^\infty Q_i)/\sum_{i=0}^M Q_i\leq \liminf_n p_n+\ep$. 
Since $\ep>0$ is arbitrary, we have $(1-\sum_{i=M+1}^\infty Q_i)/\sum_{i=0}^M Q_i\leq \liminf_n p_n$. 
Taking $M\to +\infty$, we obtain $1/\sum_{i=0}^\infty Q_i\leq \liminf_n p_n$. \\
STEP3:~If $\sum_{i=0}^\infty Q_i<+\infty$, then we have 
$1/\sum_{i=0}^\infty Q_i\leq \liminf_n p_n\leq \limsup_n p_n\leq 1/\sum_{i=0}^\infty Q_i$ 
by STEP1 and STEP2. This implies $\lim_n p_n=1/\sum_{i=0}^\infty Q_i$. 
If $\sum_{i=0}^\infty Q_i=+\infty$, then we have $\limsup_n p_n\leq 1/\sum_{i=0}^\infty Q_i=0$ by STEP1, 
so $\lim_n p_n=0=1/\sum_{i=0}^\infty Q_i$. Thus we obtain $\lim_n p_n=1/\sum_{i=0}^\infty Q_i$ in any cases. 
Since $\sum_{i=0}^\infty Q_i=\sum_{i=0}^\infty iq_i$, we complete the proof. 
\end{proof}

\begin{remark}
The idea behind the above proof is to use Tauberian theorems for almost convergent sequences. 
For simplicity, assume that $\sum_{i=0}^\infty Q_i=+\infty$. 
Let $L$ be a banach limit on $l^\infty(\Zp;\R)$. 
Let $M\geq 1$. Since $\sum_{i=0}^M Q_ip_{n-i}\leq 1~(n\geq M)$, we have 
$\sum_{i=0}^M Q_iL(p)\leq 1$, i.e., $L(p)\leq 1/\sum_{i=0}^M Q_i$. 
Since $M\geq 1$ is arbitrary, we have $L(p)\leq 1/\sum_{i=0}^\infty Q_i=0$. 
Moreover, since $p\geq 0$, we have $L(p)\geq 0$. Thus we obtain $L(p)=0$. 
Since $L$ is arbitrary, it follows that $p$ is an almost convergent sequence 
with $AC\lim_n p_n=0$. 
Since $\lim_n (p_n-p_{n-1})=0$ is a Tauberian condition for almost convergent sequences 
(e.g., \cite{lorentz1948contribution}), we obtain $\lim_n p_n=0~(=1/\sum_{i=0}^\infty Q_i)$. 
This argument itself is also a proof of Lemma \ref{dep_lemm}, 
and the above elementary proof is just a simplification of this argument. 
\end{remark}

\begin{lemma}\label{a_lemm}
If $a:\Zp\to \R$ is bounded and $\lim_n \Delta^2[a]_n=0$, then $\lim_n \Delta[a]_n=0$. 
\end{lemma}

\begin{proof}
STEP1: We first show that $\limsup_n \Delta[a]_n\leq 0$. 
By assumption, there exists $C>0$ such that $|a_n|\leq C~(n\geq 0)$. 
Let $\ep>0$ and $M\geq 1$ be arbitrary. 
Since $\lim_n \Delta^2[a]_n=0$ and $\Delta^2[a]_n=\Delta[a]_n-\Delta[a]_{n-1}~(n\geq 1)$, we have 
$\lim_n (\Delta[a]_n-\Delta[a]_{n-1})=0$, so we can easily show that $\lim_n (\Delta[a]_n-\Delta[a]_{n-i})=0~(\forall i\geq 0)$. 
Then, for any $0\leq i\leq M$, there exists $n_i\geq i$ such that $|\Delta[a]_n-\Delta[a]_{n-i}|<\ep~(n\geq n_i)$. 
Let $N:=\max\mk{M+1,n_0,n_1,\cdots,n_M}$. Then $|\Delta[a]_n-\Delta[a]_{n-i}|<\ep~(n\geq N,~0\leq i\leq M)$. 
Then $\sum_{i=0}^M \Delta[a]_{n-i}\geq \sum_{i=0}^M (\Delta[a]_n-\ep)=(M+1)(\Delta[a]_n-\ep)~(n\geq N)$. 
Moreover, $\sum_{i=0}^M \Delta[a]_{n-i}=a_n-a_{n-M-1}\leq 2C~(n\geq N)$. 
Then $2C\geq (M+1)(\Delta[a]_n-\ep)~(n\geq N)$, so $2C/(M+1)\geq \Delta[a]_n-\ep~(n\geq N)$. 
Then $\limsup_n \Delta[a]_n-\ep\leq 2C/(M+1)$. Since $\ep>0$ is arbitrary, we have 
$\limsup_n \Delta[a]_n\leq 2C/(M+1)$. Taking $M\to +\infty$, we have $\limsup_n \Delta[a]_n\leq 0$. \\
STEP2: Let $b:=-a$. We can easily verify that $b$ is bounded and $\lim_n \Delta^2[b]_n=0$. 
By STEP1, we have $\limsup_n \Delta[b]_n\leq 0$, so $\liminf_n \Delta[a]_n\geq 0$. 
Thus we obtain $\lim_n \Delta[a]_n=0$. 
\end{proof}

\begin{remark}
The idea behind the above proof is also to use Tauberian theorems for almost convergent sequences. 
Since $a$ is bounded, we can easily show that $c:=\Delta[a]$ is an almost convergent sequence 
with $AC\lim_n c_n=0$. Since $\lim_n (c_n-c_{n-1})=\lim_n \Delta^2[a]_n=0$ is a Tauberian condition 
for almost convergent sequences, we obtain $\lim_n c_n=0$, i.e., $\lim_n \Delta[a]_n=0$. 
The above  proof is just a simplification of this Tauberian argument. 
\end{remark}

\begin{corollary}
If $\lim_n \Delta[p]_n=0$ or $\lim_n \Delta^2[p]_n=0$, 
then $\lim_n p_n=1/\sum_{i=0}^\infty iq_i$. 
\end{corollary}

At this point, we have only to show that $\lim_n \Delta[p]_n=0$ or $\lim_n \Delta^2[p]_n=0$ 
to prove Theorem \ref{main_theo}. 

\section{Generating functions.}
In this section, we observe basic properties of generating functions of $q$, $p$, and $Q$. 

\begin{definition}
We define $B:=\mk{z\in \C\mid |z|<1}$ and $D:=\mk{z\in \C\mid |z|\leq 1}$. 
Next, for $a:\Zp\to \C$ and $z\in \C$, we define $f_a(z):=\sum_{n=0}^\infty a_nz^n$ 
if the right hand side converges in $\C$. 
\end{definition}

\begin{lemma}\label{qrc_lemm}
For $(r,\theta)\in (0,1]\times [-\pi,\pi]-\mk{(1,0)}$, 
we have $re^{i\theta}\in D-\mk{1}$ and $\sum_{n=0}^\infty q_nr^n\cos n\theta<1$. 
\end{lemma}

\begin{proof}
We can show this fact by using $\gcd(\mk{n\geq 1\mid q_n\neq 0})=1$. 
\end{proof}

\begin{lemma}\label{pqQ_lemm}
We have the following. \\
(i) $f_{\Delta^l[p]}(z)~(l\geq 0)$ is well-defined for $z\in B$, and continuous on $B$. \\
(ii) $f_q(z)$ is well-defined for $z\in D$, and continuous on $D$. 
Moreover, $f_p(z)(1-f_q(z))=1~(z\in B)$, $f_q(1)=1$, and $f_q(z)\neq 1~(z\in D-\mk{1})$. \\
(iii) For $(r,\theta)\in (0,1]\times [-\pi,\pi]-\mk{(1,0)}$, we have 
\begin{align}
\ak{\fr{(1-re^{i\theta})^2}{1-f_q(re^{i\theta})}}
\leq \fr{1+r^2-2r\cos\theta}{1-\sum_{n=0}^\infty q_nr^n\cos n\theta}. 
\end{align}
(iv) If $\sum_{i=0}^\infty Q_i<+\infty$, 
then $f_Q(z)$ is well-defined for $z\in D$, and continuous on $D$. 
Moreover, $(1-z)f_Q(z)=1-f_q(z)~(z\in D)$ and $f_Q(z)\neq 0~(z\in D)$. 
\end{lemma}

\begin{proof}
It is easy to show (i), (ii), and (iv). We only show (iii). 
Let $(r,\theta)\in (0,1]\times [-\pi,\pi]-\mk{(0,1)}$. 
Let $a=1-\sum_{n=0}^\infty q_nr^n \cos n\theta$ and 
$b=\sum_{n=0}^\infty q_nr^n\sin n\theta$. By Lemma \ref{qrc_lemm}, we have $a>0$. 
Moreover, we have $1-f_q(re^{i\theta})=a-ib$, so $|1/(1-f_q(re^{i\theta}))|=1/(a^2+b^2)^{1/2}\leq 1/a$. 
Thus $|(1-re^{i\theta})^2/(1-f_q(re^{i\theta}))|\leq |(1-re^{i\theta})^2|/a$. 
Since $|(1-re^{i\theta})^2|=|1-re^{i\theta}|^2=\cdots=1+r^2-2r\cos\theta$, 
we complete the proof. 
\end{proof}

\begin{lemma}\label{fourier_lemm}
For any $r\in (0,1)$ and $m,l\geq 0$, we have 
\begin{align*}
2\pi \Delta^l[p]_mr^m
=\int_{-\pi}^\pi \fr{(1-re^{i\theta})^l}{1-f_q(re^{i\theta})}e^{-im\theta}d\theta. 
\end{align*}
\end{lemma}

\begin{proof}
This is easily verified by using $f_{\Delta^l[p]}(z)=(1-z)^lf_p(z)=(1-z)^l/(1-f_q(z))~(z\in B)$. 
\end{proof}

\section{Proof of Theorem \ref{main_theo} (finite case).}

\begin{proof}
Assume that $\sum_{i=0}^\infty iq_i<+\infty$. 
We would like to show Theorem \ref{main_theo}. 
As we have already mentioned, it suffices to show that $\lim_n \Delta[p_n]=0$. 
First, since $\sum_{i=0}^\infty iq_i=\sum_{i=0}^\infty Q_i$, 
we have $\sum_{i=0}^\infty Q_i<+\infty$. By Lemma \ref{pqQ_lemm}, we have 
\begin{itemize}
\item $f_Q(z)$ is well-defined for $z\in D$, and continuous on $D$. 
Moreover, $(1-z)f_Q(z)=1-f_q(z)~(z\in D)$ and $f_Q(z)\neq 0~(z\in D)$. 
\end{itemize}
In particular, $1/f_Q(z)$ is well-defined for $z\in D$, and continuous on $D$. 
Next, by Lemma \ref{fourier_lemm} and $(1-z)f_Q(z)=1-f_q(z)~(z\in D)$, we have 
\begin{align*}
2\pi \Delta[p]_mr^m
=\int_{-\pi}^\pi \fr{1}{f_Q(re^{i\theta})}e^{-im\theta}d\theta~(0<r<1,~m\geq 0). 
\end{align*}
For any fixed $m\geq 0$, we would like to take $r\uparrow 1$. 
Since $1/f_Q(z)$ is continuous on $z\in D$, this is uniformly continuous on $D$. 
Then we have $\lim_{r\uparrow 1}e^{-im\theta}/f_Q(re^{i\theta})
=e^{-im\theta}/f_Q(e^{i\theta})$ uniformly with respect to $\theta\in [-\pi,\pi]$. Then \\
$\lim_{r\uparrow 1}\int_{-\pi}^\pi e^{-im\theta}/f_Q(re^{i\theta})d\theta
=\int_{-\pi}^\pi e^{-im\theta}/f_Q(e^{i\theta})d\theta$. 
Thus we obtain 
\begin{align*}
2\pi \Delta[p]_m=\int_{-\pi}^\pi \fr{1}{f_Q(e^{i\theta})}e^{-im\theta}d\theta~(m\geq 0). 
\end{align*}
Since $1/f_Q(e^{i\theta})$ is continuous with respect to $\theta\in [-\pi,\pi]$, 
we can apply Riemann-Lebesgue Lemma for continuous functions, so 
$\lim_m \int_{-\pi}^\pi e^{-im\theta}/f_Q(e^{i\theta})d\theta=0$, i.e, 
$\lim_m \Delta[p]_m=0$. Thus we obtain Theorem \ref{main_theo}, provided that $\sum_{i=0}^\infty iq_i<+\infty$.  
\end{proof}

\section{Proof of Theorem \ref{main_theo} (infinite case).}

\begin{proof}
Assume that $\sum_{i=0}^\infty iq_i=+\infty$. 
We would like to show Theorem \ref{main_theo}. 
As we have already mentioned, it suffices to show that $\lim_n \Delta^2[p_n]=0$. \\
STEP1: We define $G:\R\to \Rp$ as $G(x):=(1-\cos x)/x^2~(x\neq 0),~1/2~(x=0)$. 
Note that $G$ is continuous on $\R$. We also have $x^2G(x)=1-\cos x~(\forall x\in \R)$. \\
STEP2: Let $(r,\theta)\in (0,1]\times [-\pi,\pi]-\mk{(1,0)}$. 
Since $|\sin t|\leq |t|~(t\in \R)$, we have 
$1+r^2-2r\cos\theta=(1-r)^2+2r(1-\cos\theta)=(1-r)^2+4r\sin^2(\theta/2) \leq (1-r)^2+r\theta^2\leq (1-r)+\theta^2$. 
Moreover, we have 
$0<1-\sum_{n=0}^\infty q_nr^n\cos n\theta
=\sum_{n=0}^\infty q_n-\sum_{n=0}^\infty q_nr^n\cos n\theta
=\sum_{n=0}^\infty q_n(1-r^n)+\sum_{n=0}^\infty q_nr^n(1-\cos n\theta)
=(1-r)\sum_{n=0}^\infty q_n(1+r+\cdots+r^{n-1})+\theta^2\sum_{n=0}^\infty n^2q_nr^nG(n\theta)$. 
By (iii) of Lemma \ref{pqQ_lemm}, for any $(r,\theta)\in (0,1]\times [-\pi,\pi]-\mk{(1,0)}$, we obtain 
\begin{align*}
&\ak{\fr{(1-re^{i\theta})^2}{1-f_q(re^{i\theta})}}
\leq \fr{(1-r)+\theta^2}
{(1-r)\sum_{n=0}^\infty q_n(1+r+\cdots+r^{n-1})+\theta^2\sum_{n=0}^\infty n^2q_nr^nG(n\theta)}. 
\end{align*}
For any $M\geq 1$, we have $\sum_{n=0}^\infty q_n(1+r+\cdots+r^{n-1})\geq \sum_{n=0}^M q_n(1+r+\cdots+r^{n-1})$, 
so $\liminf_{(r,\theta)\to (1,0)}\sum_{n=0}^\infty q_n(1+r+\cdots+r^{n-1})\geq \sum_{n=0}^M nq_n$. 
Taking $M\to +\infty$, we obtain 
$\liminf_{(r,\theta)\to (1,0)}\sum_{n=0}^\infty q_n(1+r+\cdots+r^{n-1}\geq \sum_{n=0}^\infty nq_n=+\infty$. 
Similarly, we have $\liminf_{(r,\theta)\to (1,0)}\sum_{n=0}^\infty n^2q_nr^nG(n\theta)\geq 
(1/2)\sum_{n=0}^\infty n^2q_n\geq (1/2)\sum_{n=0}^\infty nq_n=+\infty$. Then we have 
\begin{align}\label{main_theo_pr1}
\lim_{(r,\theta)\to (1,0)}\ak{\fr{(1-re^{i\theta})^2}{1-f_q(re^{i\theta})}}=0. 
\end{align}
STEP3: We define $H:(0,1]\times [\pi,\pi]\to \C$ as follows: 
\begin{align*}
H(r,\theta):=\fr{(1-re^{i\theta})^2}{1-f_q(re^{i\theta})}~((r,\theta)\neq (1,0)),~0~((r,\theta)=(1,0)). 
\end{align*}
Then $H$ is continuous, i.e., 
$\lim_{(r,\theta)\to (r',\theta')}H(r,\theta)=H(r',\theta')~(\forall (r',\theta')\in (0,1]\times [\pi,\pi])$. 
This is easily verified for $(r',\theta')\neq (1,0)$. We also have $\lim_{(r,\theta)\to (1,0)}H(r,\theta)=H(1,0)$ 
by (\ref{main_theo_pr1}). \\
STEP4: In view of Lemma \ref{fourier_lemm} and the definition of $H$, we have 
\begin{align*}
2\pi \Delta^2[p]_mr^m=\int_{-\pi}^\pi H(r,\theta)e^{-im\theta}d\theta~(0<r<1,~m\geq 0). 
\end{align*}
For any fixed $m\geq 0$, we would like to take $r\uparrow 1$. 
Since $H(r,\theta)$ is continuous on 
$(0,1]\times [-\pi,\pi]$, this is uniformly continuous on 
$[1/2,1]\times [-\pi,\pi]$. Then we have 
$\lim_{r\uparrow 1}H(r,\theta)=H(1,\theta)$ uniformly with respect to $\theta\in [-\pi,\pi]$. Hence 
\begin{align*}
2\pi \Delta^2[p]_m=\int_{-\pi}^\pi H(1,\theta)e^{-im\theta}d\theta~(m\geq 0). 
\end{align*}
Since $H(1,\theta)$ is continuous on $\theta\in [-\pi,\pi]$, we can apply 
Riemann-Lebesgue Lemma for continuous functions, and we obtain $\lim_m \Delta^2[p]_m=0$. 
Thus we complete the proof of Theorem \ref{main_theo}. 
\end{proof}

\noindent
{\bf Acknowledgment.} \\
This research did not receive any specific grant from funding agencies 
in the public, commercial, or not-for-profit sectors. 

\noindent 
Declaration of interest statement: none. 

\bibliographystyle{plain}
\bibliography{myref}


\vfill\eject

\end{document}